\documentclass{amsart}
\usepackage[utf8]{inputenc}
\usepackage[english]{babel}

\usepackage{hyphenat}
\usepackage{xcomment}

\usepackage{amsmath,amsthm}
\usepackage{amssymb}

\usepackage[numbers]{natbib}

\usepackage{graphics}
\usepackage{caption,subcaption} 
\usepackage[all,curve]{xy} 

\usepackage{hyperref}

\usepackage{mathtools}

\newtheorem{thm}{Theorem}[section]
\newtheorem{cor}[thm]{Corollary}
\newtheorem{prop}[thm]{Proposition}
\newtheorem*{prop*}{Proposition}

\theoremstyle{definition}
\newtheorem{defn}[thm]{Definition}
\newtheorem*{defn*}{Definition}
\newtheorem{stmt}[thm]{Statement}
\newtheorem{openQuestion}{Open Question}[section]
\theoremstyle{remark}

\numberwithin{equation}{section}
\DeclarePairedDelimiter{\paren}{(}{)}
\DeclarePairedDelimiter{\abrack}{\langle}{\rangle}
\DeclarePairedDelimiter{\bracket}{[}{]}
\DeclarePairedDelimiter{\braces}{\{}{\}}

\DeclarePairedDelimiter{\abs}{\vert}{\vert}

\newcommand{\set}[1]{\braces*{#1}}
\DeclarePairedDelimiterX{\setbuild}[2]{\{}{\}}{{#1}\;\delimsize\vert\;{#2}}
\DeclarePairedDelimiterX{\setbuildc}[2]{\{}{\}}{{#1}\vcentcolon{#2}}

\DeclareSymbolFont{AMSb}{U}{msb}{m}{n}

\DeclareMathSymbol{\N}{\mathbin}{AMSb}{"4E} 
\DeclareMathSymbol{\Z}{\mathbin}{AMSb}{"5A} 
\DeclareMathSymbol{\R}{\mathbin}{AMSb}{"52} 
\DeclareMathSymbol{\Q}{\mathbin}{AMSb}{"51} 
\DeclareMathSymbol{\C}{\mathbin}{AMSb}{"43} 
\makeatletter
\newcommand{\nfs}{\@ifnextchar.{}{.\@}}
\makeatother

\newcommand{\ie}{i.e\nfs}

\newcommand{\0}{\emptyset}

\newcommand{\ce}{c.e\nfs}

\newcommand{\what}[1]{\widehat{#1}}

\newcommand{\upto}{\mathbin{\!\restriction\!}}

\def\leT{\le_{\rm T}}
\def\geT{\ge_{\rm T}}

\def\RCA0{\textbf{RCA}$_0$}
\newcommand{\concat}{%
  \mathord{
    \mathchoice
    {\raisebox{1ex}{\scalebox{.7}{$\frown$}}}
    {\raisebox{1ex}{\scalebox{.7}{$\frown$}}}
    {\raisebox{.7ex}{\scalebox{.5}{$\frown$}}}
    {\raisebox{.7ex}{\scalebox{.5}{$\frown$}}}
  }
}

\newcommand{\ud}{\overline{\rho}}

\renewcommand{\phi}{\varphi}

\title{The computational content of intrinsic density}
\author{Eric P. Astor}
\date{\today}


\begin{document}
\begin{abstract}
In a previous paper, the author introduced the idea of intrinsic density --- a restriction of asymptotic density to sets whose density is invariant under computable permutation. We prove that sets with well-defined intrinsic density (and particularly intrinsic density 0) exist only in Turing degrees that are either high ($\mathbf{a}'\geT\0''$) or compute a diagonally non-computable function. By contrast, a classic construction of an immune set in every non-computable degree actually yields a set with intrinsic lower density 0 in every non-computable degree.

We also show that the former result holds in the sense of reverse mathematics, in that (over $\mathbf{RCA}_0$) the existence of a dominating or diagonally non-computable function is equivalent to the existence of a set with intrinsic density 0.
\end{abstract}

\maketitle

\section{Introduction}

Shortly after the launch of the field of computability, practitioners began exploring the connections between the computability of a set and the scarcity of its elements. Post, seeking a non-computable \ce{} set with less computational power than the halting problem, based his approach on the idea of creating sets with increasingly thin complement, making their complements computationally difficult to distinguish from finite. Those following his program developed the classical ``immunity hierarchy'' of thinness properties, from immune sets to cohesive sets. (See Figure~\ref{fig:immunityHierarchy}.) Though none of these exhibited the strict upper bound Post sought, lower bounds on the Turing degrees of levels of this hierarchy describe useful dividing lines in computational content. For example, immune sets exist precisely in every non-computable degree, while the hyperimmune-free degrees are those that contain only computably-bounded functions. Moving upwards in the hierarchy, cohesiveness and other forms of immunity force co-\ce{} sets to be high ($\mathbf{a}'\geT\0''$), while revealing much more complex patterns outside of the $\Delta^0_2$ degrees. \cite{nonhighCohesive}

\begin{figure}[t]
\centering
\hfill%
\begin{subfigure}{.35\linewidth}
\centering
\resizebox{\linewidth}{!}{
\xymatrix@R-8pt@C=-15pt{
	& & *+[F]+\hbox{Cohesive} \ar[dl] \ar[dr] &\\
	& *+[F]+\hbox{q-Cohesive} \ar[d] & & *+[F]+\hbox{r-Cohesive} \ar[d]\\
	& *+[F]+\hbox{shh-Immune} \ar[d] \ar[rr] & & *+[F]+\hbox{sh-Immune} \ar[d]\\
	& *+[F]+\hbox{hh-Immune} \ar[drr] & & *+[F]+\hbox{fsh-Immune} \ar[d]\\
	& *+[F]+\hbox{Dense Immune} \ar[rr] & & *+[F]+\hbox{Hyperimmune} \ar[d]\\
	& & & *+[F]+\hbox{Immune}
	}
}
\subcaption{For general sets}
\end{subfigure}%
\hfill%
\begin{subfigure}{.35\linewidth}
\centering
\resizebox{\linewidth}{!}{
\xymatrix@R-8pt@C=-15pt{
	& *+[F]+\hbox{Cohesive} \ar[dl] \ar[dr] & \\
	*+[F]+\hbox{q-Cohesive} \ar[d] & & *+[F]+\hbox{r-Cohesive} \ar[d]\\
	*+[F]+\hbox{(s)hh-Immune} \ar[dd] \ar[rr] & & *+[F]+\hbox{sh-Immune} \ar[d]\\
	& & *+[F]+\hbox{fsh-Immune} \ar[d]\\
	*+[F]+\hbox{Dense Immune} \ar[rr] & & *+[F]+\hbox{Hyperimmune} \ar[d]\\
	& & *+[F]+\hbox{Immune}
	}
}
\subcaption{For co-\ce{} sets}
\end{subfigure}
\hspace*{\fill}%
\caption{\cite{dirPaper} The graphs of the implications between the classical immunity properties; for $\Delta^0_2$ sets, the implications are the same as in the general case, except that shh-immunity and hh-immunity become equivalent. All implications are strict, and any not shown (excepting those implied by transitivity) are false.}
\label{fig:immunityHierarchy}
\end{figure}
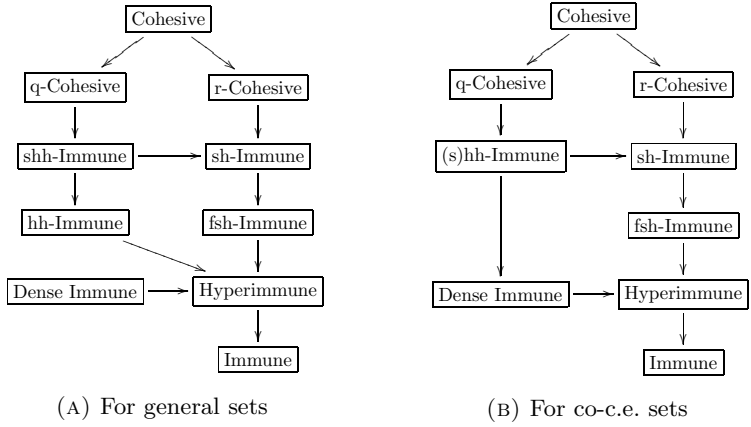

More recently, \citet*{JSgc} (inspired by work of \citet*{genericComplexity} on decidability in group theory) constructed new notions of near-computation, considering computations modulo sets with asymptotic density 0. This approach of computability modulo sparse sets, joined by other researchers (including Downey \cite{DJSdensity,ershovDensity}, Dzhafarov \cite{robustCoding}, Hirschfeldt \cite{upperCones,coarseReducibility,coarseBound}, Igusa \cite{robustCoding,igusaNoMinimalPair}, and McNicholl \cite{ershovDensity,coarseBound}), has uncovered still more connections between thinness and computation. However, since asymptotic density is not invariant under computable permutation, defining thinness in terms of density gives a notion that is incomparable to the standard immunity properties, and generally ill-behaved from the perspective of computability theory.

In a previous paper \cite{dirPaper}, the author suggested that we instead consider sets to be thin if no computable process can sample the set at positive density infinitely often (\ie, with positive upper density); we say such a set has \emph{intrinsic density 0}. More precisely:
\begin{defn}
If $S$ is a subset of the natural numbers, we say that $S$ has intrinsic density 0 if
\begin{equation*}
\lim_{n\to\infty}{\frac{\abs{\pi(S)\upto n}}{n}}=0
\end{equation*}
for all computable permutations $\pi:\omega\to\omega$.
\end{defn}
By Corollary~4.3 of \cite{dirPaper}, this notion remains unchanged if, rather than taking images of $S$ under computable permutations, we take preimages under computable injections; that is,
\begin{prop}
A set $S\subseteq\omega$ has intrinsic density 0 if and only if
\begin{equation*}
\lim_{n\to\infty}{\frac{\abs{p^{-1}(S)\upto n}}{n}}=0
\end{equation*}
for every computable injection $p:\omega\to\omega$.
\end{prop}

We can weaken this slightly, instead asking only that no sampling succeeds with positive density in the limit (\ie, with positive lower density).
\begin{defn}
A set $S\subseteq\omega$ has \emph{intrinsic lower density 0} if
\begin{equation*}
\liminf_{n\to\infty}{\frac{\abs{p^{-1}(S)\upto n}}{n}}=0
\end{equation*}
for every computable injection $p:\omega\to\omega$.
\end{defn}

Both of these properties imply immunity, are weaker than cohesiveness, and are preserved under taking infinite subsets; that is to say, they are new immunity properties, fitting neatly into the classical hierarchy, as shown in Figure~\ref{fig:densityImmunityHierarchy}.

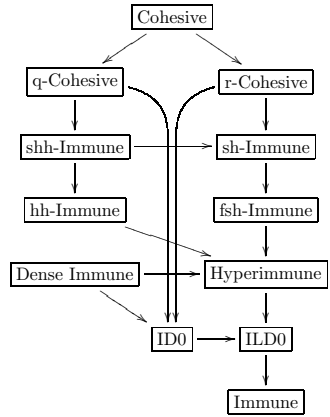
\begin{figure}
\centering
\resizebox{0.35\linewidth}{!}{
\xymatrix@R-8pt@C=-10pt{
	& *+[F]+\hbox{Cohesive} \ar[dl] \ar[dr] & \\
	*+[F]+\hbox{q-Cohesive} \ar[d] \ar@`{[0,1]-<0.5ex,0pt>,[1,1]-<0.5ex,0pt>,[2,1]-<0.5ex,0pt>}[4,1]!<-0.5ex,0pt> & & *+[F]+\hbox{r-Cohesive} \ar[d] \ar@`{[0,-1]+<0.5ex,0pt>,[1,-1]+<0.5ex,0pt>,[2,-1]+<0.5ex,0pt>}[4,-1]!<0.5ex,0pt>\\
	*+[F]+\hbox{shh-Immune} \ar[d] \ar[rr] & & *+[F]+\hbox{sh-Immune} \ar[d]\\
	*+[F]+\hbox{hh-Immune} \ar[drr] & & *+[F]+\hbox{fsh-Immune} \ar[d]\\
	*+[F]+\hbox{Dense Immune} \ar[dr] \ar[rr] & & *+[F]+\hbox{Hyperimmune} \ar[d]\\
	& *+[F]+\hbox{ID0} \ar[r] & *+[F]+\hbox{ILD0} \ar[d]\\
	& & *+[F]+\hbox{Immune}
	}
}
\caption{\cite{dirPaper} The graph of implications between the classical immunity properties and intrinsic density 0. Again, for $\Delta^0_2$ sets, shh-immunity and hh-immunity become equivalent; all other implications are as depicted. (We abbreviate intrinsic [lower] density 0 for infinite sets by I[L]D0.)}
\label{fig:densityImmunityHierarchy}
\end{figure}

In Section~\ref{sec:idComputability}, we establish tight lower bounds on the Turing degrees of sets with these new properties. By careful analysis of a classic construction of immune sets via introreducibility, we prove that sets with intrinsic lower density 0 exist in every non-computable Turing degree. By contrast, sets with intrinsic density 0 have more interesting computational content; we find that the Turing degrees of these sets are precisely those that are high ($\mathbf{a}'\geT\0''$) or compute a diagonally non-computable function ($\mathbf{a}\geT f$, where $f(e)\ne\phi_e(e)$ for every $e$). For this section, we assume familiarity with basic computability theory.

In Section~\ref{sec:idReverseMath}, we analyze the latter result from the perspective of reverse mathematics. (Naturally, this section will additionally assume the reader's familiarity with reverse mathematics, as formulated in \citet*{SOSOA}.) Formalizing our principles by means of a framework developed by \citet*{weaklyRepresented}, we find that our proofs regarding intrinsic density 0 all hold over $\mathbf{RCA}_0$, and so show that the existence of a set with intrinsic density 0 is equivalent to a disjunction of two more standard reverse-mathematical principles.

We adopt conventions typical in computability, generally following those established in \citet*{cta}. We refer to the set of natural numbers as $\omega$, and denote the $e$-th partial computable function (in some effective enumeration) by $\phi_e$ (while referring to the universal Turing functional as $\Phi$). We routinely identify a set $S\subseteq\omega$ with its characteristic function, as well as with the infinite binary sequence $(S(n))_{n\in\omega}$. We will also make common use of the notation $S\upto n$, to be read ``the set $S$ up to $n$'', and interpreted as $S\cap[0,n)$ --- most often identified with the $n$-bit prefix of the infinite binary sequence just discussed.

\section{Intrinsic density in the Turing degrees}\label{sec:idComputability}

As mentioned above, bounds (and particularly lower bounds) on the Turing degrees of immunity properties have historically yielded useful dividing lines in the sense of computational content; immunity matches with non-computability, hyperimmunity with escaping computable bounds, and cohesiveness with a more complex class of Turing degrees that, when restricted to the $\Delta^0_2$-setting, coincides with highness. The new entries in the hierarchy are no exception.

Recall that a set is immune if it has no infinite computable (or \ce{}) subset. One classic proof that every set is Turing equivalent to an immune set is as follows: consider a set $A$ as an infinite binary sequence, and let $S$ be the set of its finite prefixes (clearly Turing equivalent to $A$). $S$ is introreducible (\ie, computable from any infinite subset of itself), since if we are given infinitely many elements of $S$, we can compute every bit of $A$, and thus can recover $S$ in full. In particular, $S$ is immune if and only if $A$ is not computable; if $S$ were not immune, it would have an infinite computable subset, from which we could compute all of $S$ (and thus $A$).

With more careful analysis, we can in fact show the slightly stronger fact that $S$ has intrinsic lower density 0.

\begin{thm}
Let $S$ be the set of prefixes of a set $A$. If $A$ is not computable, then $S$ has intrinsic lower density 0.
\end{thm}
\begin{proof}
Suppose $S=\setbuildc{A\upto n}{n\in\N}$, the set of prefixes of $A$, does not have intrinsic lower density 0. By definition, there exist integers $q$ and $N$, and some total computable injection $\phi_e$, such that $\rho_n(\phi_e^{-1}(S))>\frac{1}{q}$ for all $n>N$.

Given $e$, $q$, and $N$, we construct the computable binary tree $T$ as follows:

$T$ begins as a full tree up to height $N$. For strings $\sigma$ of length $n>N$, we put $\sigma$ into $T$ if and only if its prefixes are in $T$ and $\phi_e(\left[0,2qn\right))$ contains at least $n$ strings extending $\sigma$.
As no two distinct strings of the same length can share an extension, and since $\sigma\in T$ implies that $\phi_e(\left[0,2q\abs{\sigma}\right))$ contains at least $\abs{\sigma}$ extensions of $\sigma$, we see that $T$ has width at most $2q$ at all heights $n>N$.

By assumption, $\phi_e$ samples the prefixes of $A$ with partial density always exceeding $\frac{1}{q}$ beyond a point $N$; therefore, for $n>N$, $\phi([0,2qn))$ must contain at least $2n$ prefixes of $A$, and so must include at least $n$ extensions of $A\upto n$. Thus, $A$ must be a path on $T$.

Since $T$ is a computable tree with bounded width, $A$ is computable.
\end{proof}

By Corollary~1.4 of \cite{dirPaper}, no infinite computable set can have intrinsic lower density 0; therefore,

\begin{cor}
The Turing degrees containing an infinite set of intrinsic lower density 0 (or, taking the complement, a co-infinite set of intrinsic upper density 1) are precisely the non-computable degrees.
\end{cor}

As for intrinsic density 0, previous work by the author \cite{dirPaper} has already established some upper bounds on the information content required to compute such sets. As every r-cohesive set has intrinsic density 0, these sets exist in every cohesive degree, and in particular in every high degree. Moreover, sets with intrinsic density 0 can be computed from any 1-random, and thus from any 1-random degree. These bounds (taken together) prove to be nearly tight.

To refine these bounds, we rely on a characterization by \citet*{complexityRecursion} of the degrees of eventually-different functions (to be discussed more in Section~\ref{sec:idReverseMath}). Here, we focus on the complementary notion, under the name introduced in \citet*{dhBook}.

\begin{defn}
A set $A$ is \emph{weakly computably traceable} if there is a computable function $h$ such that for all $f\leT A$, there is a computable sequence of finite sets $V_n$ with $\abs{V_n}\le h(n)$ for all $n$ and $f(n)\in V_n$ for infinitely many $n$; that is, if we can infinitely often guess the value of $f(n)$ using a computable guessing strategy limited to at most $h(n)$ guesses.
\end{defn}

By Theorem~5.1 of \citet*{complexityRecursion}, a set $A$ is weakly computably traceable (WCT) iff it has neither high nor DNC degree; that is, iff it computes neither a dominant function \cite{martin66} nor a diagonally non-computable function. Thus, weak computable traceability is a property of Turing degrees, expressing computability-theoretic weakness.

Since all 1-random sets compute a DNC function, and all r-cohesive sets have either high or DNC degree \cite{nonhighCohesive}, all of our prior constructions of a set with intrinsic density 0 were built below non-WCT sets. We can now show that this was no coincidence.

\begin{thm}\label{thm:wctDensity}
Every infinite set $A$ that is weakly computably traceable has upper density 1 under some computable sampling, and thus has absolute upper density 1.
\end{thm}
\begin{proof}
By the same equivalence of \citet*{complexityRecursion}, $A$ is weakly computably traceable iff for all $f\le_T A$, there is a total computable function $h$ such that $h(n)=f(n)$ for infinitely many $n$.

Let $f(n)$ code $p_A(j)$ for all $j<n!$; specifically, take $f(n)=A\upto p_A(n!)$. Clearly $f\le_T A$, so there is a total computable $h$ with $h(n)=f(n)$ for infinitely many $n$.

We define a total computable injection $g$ by assigning values $g(j)$ in increasing order of $j$. If $j\in\left[(n-1)!,n!\right)$, define $g(j)$ to be the position of the $j$-th 1 in the string $h(n)$, unless this value is already assigned to some $g(i)$ with $i<j$; in that case, we instead define $g(j)$ to be the least value not assigned to any earlier $g(i)$.

For any $n$ where $h(n)=f(n)$, we then have $g(j)\in A$ for all $j\in\left[(n-1)!,n!\right)$, unless the requisite value was already assigned at that stage. In any event, $g(\left[0,n!\right))$ contains at least $n!-(n-1)!$ elements of $A$, so $\rho_{n!}(g^{-1}(A))\ge 1-\frac{1}{n}$. Since this occurs for infinitely many $n$, we conclude that $g$ samples $A$ with upper density 1, and thus (by Lemma~4.2 of \cite*{dirPaper}) that $A$ has absolute upper density 1.
\end{proof}

\begin{cor}\label{cor:wctAbsolute}
If a set $A$ is infinite, co-infinite, and weakly computably traceable, then $A$ has absolute upper density 1 and absolute lower density 0.
\end{cor}
\begin{proof}
Apply Theorem~\ref{thm:wctDensity} to both $A$ and $\overline{A}$.
\end{proof}

By generalizing our prior constructions, we can prove this new bound tight.

\begin{thm}\label{thm:id0Construction}
Every set $A$ that is not weakly computably traceable computes a set with intrinsic density 0.
\end{thm}
\begin{proof}
By definition, $A$ is not weakly computably traceable iff for all computable orders $h$, there is some $f\le_T A$ such that for no computable function $g(n)$ do we have $\abs{D_{g(n}}\le h(n)$ for all $n$ and $f(n)\in D_{g(n)}$ infinitely often.

Take $h(n)=n^2$ (or, indeed, any computable superlinear function), and let $f\le_T A$ be as above. We claim that $G_f=\setbuildc{\abrack{n,f(n)}}{n\in\N}$, the graph of $f$, has intrinsic density 0.

Suppose not; in particular, suppose that $G_f$ has upper density greater than $\frac{1}{q}$ under some computable permutation $\pi$. Thus, $G_f$ has partial density exceeding $\frac{1}{q}$ in the first $s$ positions for infinitely many $s$. For such $s$, we have that $\pi(\left[0,s\right))$ contains at least $\frac{s}{q}$ elements of $G_f$, and thus must contain $\abrack{m,f(m)}$ for some $m\ge\frac{s}{q}-1$. Therefore, for infinitely many $m$, we have that $\pi(\left[0,(m+1)q\right))$ contains $\abrack{m,f(m)}$.

For all $n$, define $D_{g'(n)}=\setbuildc{y}{\abrack{x,y}\in\pi(\left[0,(n+1)q\right))}$. For all sufficiently large $n$, $h(n)\ge(n+1)q$; thus, there is a computable function $g$ such that $D_{g(n)}\le h(n)$ for all $n$ and $g(n)=g'(n)$ for all sufficiently large $n$.

However, as noted above, $\pi(\left[0,(m+1)q\right))$ contains $\abrack{m,f(m)}$ for infinitely many $m$. As this implies that $f(m)\in D_{g(m)}$ for infinitely many $m$, this contradicts our choice of $f$. Therefore, we conclude that $G_f$ has intrinsic density 0.
\end{proof}

By a remarkable result of \citet*{upwardClosure}, the class of Turing degrees containing infinite sets of intrinsic density 0 is upwards closed; we need only note that there are arithmetic sets of intrinsic density 0, and that intrinsic density 0 is preserved under taking subsets. Combining this observation with Corollary~\ref{cor:wctAbsolute} and Theorem~\ref{thm:id0Construction}, we obtain the following corollary:

\begin{cor}
The Turing degrees containing an infinite set of intrinsic density 0 are precisely those that are not weakly computably traceable; that is, those that are either high or DNC.
\end{cor}

Since all weakly computably traceable sets have absolute upper density 1 by Theorem~\ref{thm:wctDensity}, this gives a 0-1 law for absolute upper density:

\begin{cor}
A Turing degree contains no set with absolute upper density 0 iff all its sets have absolute upper density 1.
\end{cor}

In fact, our results are slightly broader than stated above, since Corollary~\ref{cor:wctAbsolute} states that no non-trivial weakly computably traceable set has any defined intrinsic density.

\begin{cor}\label{cor:idLowerBound}
The Turing degrees containing infinite co-infinite sets with defined intrinsic density are precisely those that are not weakly computably traceable; that is, those that are either high or DNC.
\end{cor}

By Arslanov's completeness criterion, any DNC \ce{} set is in fact Turing\hyp{}equivalent to $\0'$; therefore, each non-high co-infinite \ce{} set has absolute lower density 0, as its complement is weakly computably traceable.

On the other hand, since there is a dense simple (in fact, maximal) set in every high \ce{} degree \cite{martin66}, every high \ce{} degree contains a \ce{} set with intrinsic density 1. \cite{dirPaper} Therefore, the Turing degrees computing a co-infinite \ce{} set of intrinsic density 1 are precisely the high \ce{} degrees.

Since there are non-high hypersimple sets (in fact, every non-computable \ce{} degree contains a hypersimple set), this answers an open question from \cite{dirPaper}:

\begin{cor}\label{cor:hypersimpleLD}
There is a hypersimple set with lower density 0.
\end{cor}
\begin{proof}
Let $A$ be any non-high hypersimple set. Since $A$ is \ce{} and non-high, it is weakly computably traceable, and therefore has lower density 0 under some computable sampling. Taking the image of $A$ under the computable permutation mentioned in Lemma~4.2 of \cite*{dirPaper}, and noting that hypersimplicity is computably invariant, we obtain a hypersimple set with lower density 0.
\end{proof}

In this section, we have characterized the Turing degrees of sets with intrinsic density 0, and (in Corollary~\ref{cor:idLowerBound}) noted that of all sets with well-defined intrinsic density, these are the ``easiest'' compute. However, this only provides a lower bound on the Turing degrees of sets with intrinsic density intermediate between 0 and 1, leaving us with an open question:

\begin{openQuestion}\label{q:intermediateIDdegrees}
What are the Turing degrees of sets with intermediate intrinsic density (for concreteness, density~$\frac{1}{2}$)?
\end{openQuestion}

As shown in \citet*{dirPaper}, intrinsic density~$\frac{1}{2}$ is notable as a weak notion of randomness, coinciding with permutation and injective stochasticity (by analogy to the permutation and injective randomness of \citet*{permutationRandomness}).

Schnorr randomness provides an upper bound: as all Schnorr random sets have density~$\frac{1}{2}$, and Schnorr randomness is computably invariant, every Schnorr random must have intrinsic density~$\frac{1}{2}$. By \citet*{randomDegrees}, Schnorr random sets exist in every high or 1-random degree. This leaves us with a gap. Corollary~\ref{cor:idLowerBound} tells us that every non-high set with intrinsic density~$\frac{1}{2}$ computes a DNC function, but:

\begin{openQuestion}
Does every non-high set with intrinsic density~$\frac{1}{2}$ compute a 1-random set?
\end{openQuestion}

A positive answer to this would resolve Question~\ref{q:intermediateIDdegrees}, showing that the degrees of sets with intrinsic density~$\frac{1}{2}$ coincide with the Schnorr random degrees. (We should note that not every set with intrinsic density~$\frac{1}{2}$ is itself Schnorr random; one can show that if $A$ has intrinsic density~$\frac{1}{2}$, so does $A\oplus A$.)

\section{Intrinsic density in reverse mathematics}\label{sec:idReverseMath}

We note that the proofs of Theorems~\ref{thm:wctDensity} and \ref{thm:id0Construction}, as given above, appear sufficiently constructive to hold in $\mathbf{RCA}_0$. This suggests that, in the sense of reverse mathematics, the existence of a non-weakly-computably-traceable set should imply the existence of a set with intrinsic density 0, which should in turn imply the existence of a function that cofinitely differs from every computable function (sometimes called an \emph{eventually different} function). Presuming that Kjos-Hanssen, Merkle, and Stephan's Theorem~5.1 \cite*{complexityRecursion} holds in $\mathbf{RCA}_0$, this should give a full reverse-mathematical equivalence. Using another part of their theorem, these principles should also be equivalent to the existence of either a dominating or DNR function, extending our computability-based characterization to a fact of reverse mathematics.

To make this more precise, all of these principles need to be formally specified. Fortunately, in their investigation of cardinal invariants and reverse mathematics, \citet*{weaklyRepresented} have developed a framework convenient for this purpose, and have already formalized the existence of an eventually-different function. We recall their definitions:

\begin{defn}[Weakly-represented partial functions]
A partial function $f$ is weakly represented by the set $A$ if all of the following conditions hold:
\begin{itemize}
\item{} [Representation] For all $x$ and $y$, $f(x)\downarrow=y$ iff there is some $z$ such that $\abrack{x,y,z}\in A$. We say $A$ \emph{witnesses} that $f(x)$ converges to $y$.
\item{} [Consistency] If $\abrack{x,y,z}$ and $\abrack{x,y',z'}$ are both in $A$, then $y=y'$.
\item{} [Monotonicity] If $\abrack{x,y,z}\in A$, then for all $z'>z$, we also have $\abrack{x,y,z}\in A$.
\item{} [Downward closure] If $A$ witnesses that $f(x)$ converges, then it also witnesses that $f(t)$ converges for all $t<x$.
\end{itemize}

By convention, for $f$ weakly represented by $A$, we say that \emph{$f(x)$ converges to $y$ by step $z$} ($f(x)[z]\downarrow=y$) if $y<z$ and $\abrack{x,y,z}\in A$. Along the same line, we say that $f(x)[s]$ converges iff it converges to some $y<s$; this restriction ensures that the question of whether $f(x)[s]$ converges is decidable in our representation of $f$.
\end{defn}

For the remainder of this section, we make common reference to an arbitrary model of second-order arithmetic,
\begin{equation*}
\mathcal{M}=\abrack*{M,S,+,\cdot,0,1},
\end{equation*}
where $M$ and $S$ are, respectively, the first- and second-order parts of the structure.

\begin{defn}[Weakly-represented families]
A class of partial functions $\set{f_e}_{e\in M}$ is weakly represented in our model iff $S$ contains a uniform family of sets $\set{A_e}_{e\in M}$ (represented by $A=\setbuildc{\abrack*{e,x}}{x\in A_e}\in S$) such that $A_e$ weakly represents $f_e$.

A class of total functions $\mathcal{F}$ is weakly represented in our model iff $S$ contains $\mathcal{F}$ and a weakly-represented class of partial functions $\set{f_e}_{e\in M}$ such that a total function $f$ is in $\mathcal{F}$ iff $f=f_e$ for some $e\in M$.
\end{defn}

Restricting ourselves to 0-1 functions in the latter case naturally provides the idea of a weakly-represented family of sets.

These definitions enable us to discuss the subset of total functions within a larger class of partial functions. For instance, the family of all computable functions (or sets) is weakly representable in $\mathbf{RCA}_0$.

H\"olzl, Raghavan, Stephan, and Zhang defined these notions to formulate, as reverse-mathematical principles, the many concepts from classical computability theory which naturally address the class of total functions, such as dominating functions (their $\mathbf{DOM}$) or cohesive sets ($\mathbf{COHW}$); they also introduce a principle of particular relevance to us, the existence of an eventually-different (or, in their terminology, avoiding) function ($\mathbf{AVOID}$).

\begin{stmt}[$\mathbf{DOM}$]
For every weakly-represented family of total functions $\mathcal{F}$, there is a function $g$ such that, for each $f\in\mathcal{F}$, there is some $b\in M$ such that $g(x)>f(x)$ for all $x>b$.
\end{stmt}

\begin{stmt}[$\mathbf{COHW}$]
For every weakly-represented family of sets $\mathcal{F}$, there exists an $\mathcal{F}$-cohesive set.
\end{stmt}

\begin{stmt}[$\mathbf{AVOID}$]
For every weakly-represented family of total functions $\mathcal{F}$, there is a function $g$ such that for each $f\in\mathcal{F}$, the set $\setbuildc{x\in M}{f(x)=g(x)}$ is bounded.
\end{stmt}

In this vein, we can now state the existence of a set of intrinsic density 0 as a reverse\hyp{}mathematical principle:

\begin{stmt}[$\mathbf{ID0}$]
For every weakly-represented class of total functions $\mathcal{F}$, there exists a set $A$ such that every injective $f\in\mathcal{F}$ samples $A$ with density 0. That is, taking $\what{\mathcal{F}}$ to be the weakly-represented class of total injections contained in $\mathcal{F}$,
\begin{equation*}
\paren*{\exists A}\paren*{\forall f\in\what{\mathcal{F}}}\bracket*{\rho(f^{-1}(A))=0}.
\end{equation*}
\end{stmt}

We can similarly give an alternate form of $\mathbf{DNR}$, equivalent to the standard form over $\textbf{RCA}_0$:

\begin{stmt}[$\mathbf{DNRW}$]
For every weakly-represented family of (partial) functions $\mathcal{F}=\set{f_e}_{e\in M}$, there exists a function $F$ such that $F(e)\ne f_e(e)$ for all $e\in M$ such that $f_e(e)\downarrow$.
\end{stmt}

\begin{stmt}[$\mathbf{DNR}$]
For every set $A$, there is a function $f:M\to M$ that is diagonally non-recursive in $A$; that is, for all $x\in M$, $\Phi^A_x(x)$ does not converge to $f(x)$.
\end{stmt}

\begin{thm}\label{thm:weakDNR}
Over $\mathbf{RCA}_0$, $\mathbf{DNR}$ and $\mathbf{DNRW}$ are equivalent.
\end{thm}
\begin{proof}
To see that $\mathbf{DNRW}$ implies $\mathbf{DNR}$, we note that though we cannot weakly represent the standard listings of partial functions $\set{\phi^A_e}$ (since $\phi^A_e$'s domain need not be an initial segment of $M$), we can weakly represent the family of partial functions $\set{f^A_e}$ given by $f^A_e(n)=\phi_e(e)$. Any function diagonally disagreeing with $\set{f^A_e}$ must in fact be $\mathrm{DNR}$ relative to $A$.

The converse implication is also relatively straightforward. Fixing some weakly\hyp{}represented family of partial functions $\set{f_e}$ uniformly computable from $A$, we pass to another family, still uniformly computable from $A$:
\begin{align*}
g_{2e}(2n)=g_{2e}(2n+1)&=f_e(n),\\
g_{2e+1}(n)&=\phi^A_e(n).
\end{align*}
Since the $f_e$'s were uniformly computable from $A$, this family $\set{g_k}$ is an $\emph{effective}$ universal listing of partial $A$-computable functions; therefore, by $\mathbf{DNR}$, there is some $G$ such that if $g_k(k)\downarrow$, then $G(k)\ne g_k(k)$. In particular, if $f_e(e)=g_{2e}(2e)\downarrow$, then $G(2e)\ne f_e(e)$. Thus, defining $F(e)=G(2e)$, we see that our original family $\set{f_e}$ satisfies $\mathbf{DNRW}$ via $F$.
\end{proof}

Given our interest in the existence of either a DNR or dominating function, we formulate this disjunction as a principle as well:
\begin{stmt}[$\mathbf{DNRW}\lor\mathbf{DOM}$]
For every weakly-represented family of (partial) functions $\mathcal{F}=\set{f_e}_{e\in M}$, there exists either a function $F$ dominating all total functions in $\mathcal{F}$ or a function $g$ such that $g(e)\ne f_e(e)$ for all $e\in M$.
\end{stmt}
All of our proofs would hold if we restricted ourselves to classes containing only total functions, avoiding the complication of weak representation --- but in this case, our results would be trivial, as all of these simplified principles are true in $\mathbf{RCA}_0$.

Our proof of Theorem~\ref{thm:wctDensity} in fact shows that $\mathbf{ID0}$ implies $\mathbf{AVOID}$ over $\mathbf{RCA}_0$, using no additional assumptions or induction.

Since the existence of a non-weakly-computable-traceable set does not lend itself to formalization over $\mathbf{RCA}_0$, we instead work through another part of Theorem~5.1 of \citet*{complexityRecursion}; in particular, the part of their proof labeled ``(1) implies (2)'' shows that $\mathbf{AVOID}$ implies $\mathbf{DNRW}\lor\mathbf{DOM}$, using only methods available in $\mathbf{RCA}_0$.

It therefore suffices to show that $\mathbf{DNRW}\lor\mathbf{DOM}$ implies $\mathbf{ID0}$ over $\mathbf{RCA}_0$. This can be checked by taking the proof of our Theorem~\ref{thm:id0Construction} together with three of the subproofs from Kjos-Hanssen, Merkle, and Stephan's Theorem~5.1 \cite{complexityRecursion}, factoring out the intermediate steps, and verifying all remaining claims in $\mathbf{RCA}_0$. However, this process is quite involved, as one could easily lose track of a dependence on something not present in $\mathbf{RCA}_0$; for the reader's reference, we therefore include the full proof we obtained by this method.

\begin{thm}\label{thm:DNRorDOMid0}
$\mathbf{RCA}_0+(\mathbf{DNRW}\lor\mathbf{DOM})\models\mathbf{ID0}$.
\end{thm}
\begin{proof}
Consider some weakly-represented family of (partial) functions $\mathcal{F}=\set{f_e}$; without loss of generality, we assume $\mathcal{F}$ is universal. By $\mathbf{DNRW}\lor\mathbf{DOM}$, there exists either a function $F$ dominating all total functions in $\mathcal{F}$ or a function $f$ with $f(n)\ne f_n(n)$ for all $n\in M$.

Suppose that there is a function $F$ dominating all total functions in $\mathcal{F}$; without loss of generality, we may assume $F$ to be strictly increasing. We define $I$ to be the image of $F$:
\begin{equation*}
I=\setbuildc{n\in M}{\paren{\exists s}\bracket*{f(s)=n}}.
\end{equation*}
Suppose there is some total injective $f_e\in\mathcal{F}$ sampling $I$ with positive upper density. We then choose some $q\in M$ with $\ud(f_e^{-1}(S))>\frac{1}{q}$, and define $h(n)=1+\max_{s\le(n+1)q}{f_e(s)}$.

By our choice of $f_e$, there is an unbounded set of $s$'s such that $f_e(\left[0,s\right))$ contains at least $\frac{s}{q}$ values of $F$, and thus includes $F(m)$ for some $m\ge\frac{s}{q}-1$. Therefore, there is an unbounded set of $n$'s such that $g(\left[0,(n+1)q\right))$ contains $F(n)$.

For each such $n$, we have $h(n)>F(n)$. Since $F$ is a dominating function for all functions in $\mathcal{F}$, this implies that $h(n)\not\in\mathcal{F}$. However, $h\leT g\in\mathcal{F}$, contradicting our assumption that $\mathcal{F}$ is universal. Therefore, if every total function in $\mathcal{F}$ is dominated by $F$, every total injection in $\mathcal{F}$ samples $I$ with density 0.

On the other hand, suppose instead that there is a function $g$ with $g(n)\ne f_n(n)$ for all $n\in M$. Fix a universal-for-$\mathcal{F}$ machine $U$, and let $\Psi$ be the computable functional such that
\begin{equation*}
\Psi^X(x)=\paren{\mu y}\bracket*{\abrack*{x,y}\in X}
\end{equation*}
if for all $x'\le x$, there exists some $y'$ such that $\abrack*{x',y'}\in X$; otherwise, $\Psi^X(x)\uparrow$.

We then define $e(\sigma)$ (for any $M$-finite binary string $\sigma$) such that $f_{e(\sigma)}=\Psi^{U(\sigma)}$ (possible, since $\mathcal{F}$ is universal), and let
\begin{equation*}
p(n)=1+\max_{\abs{\sigma}<5\log{n}}{\abrack*{e(\sigma),g(e(\sigma))}}.
\end{equation*}
Since $p\leT g$, $p$ exists in $\mathcal{M}$.

Using this, we define the set $P=\set{A\upto p(n)}_{n\in M}$, where $A=G(g)$; we will show that every total injective $f_e\in\mathcal{F}$ samples $P$ with density 0.

Suppose, for the sake of contradiction, that there is some $s$ such that $f_e$ is a total injection sampling $P$ with positive upper density; specifically, take this upper density to be greater than $\frac{1}{q}$. We define $h$ so that $h(n)$ codes the $M$-finite set $D_{h(n)}=f_e(\left[0,n^2\right))$; clearly, $h\leT f_e\leT\mathcal{F}$.

By assumption, there is an unbounded set of $s$'s such that $f_e(\left[0,s\right))$ contains at least $\frac{s}{q}$ elements of $P$, including $(A\upto p(m))$ for some $m\ge\frac{s}{q}-1$. Therefore, there is an unbounded set of $n$'s such that $f_e(\left[0,(n+1)q\right))$ contains $(A\upto p(n))$. Since $n^2\ge(n+1)q$ for all sufficiently large $n$, there is an unbounded set of $n$'s such that $(A\upto p(n))\in D_{h(n)}$.

Take $k(n)$ to be the 2-to-1 prefix-free binary coding of $n\in M$ (with end symbol) in $2\log{n}+2$ bits. For $x<n^2$, let $c_n(x)$ be the standard binary coding of $x$ in $2\log{n}$ bits. Since $U$ is a universal machine for $\mathcal{F}$ and $h\leT\mathcal{F}$, there is a string $\sigma$ such that, for $x<n^2$, $U(\sigma\concat k(n)\concat c_n(x))$ is the $x$-th element of $D_{h(n)}$.

For all $n$ such that $(A\upto p(n))\in D_{h(n)}$, we define $\sigma_n=\sigma\concat k(n)\concat c_n(x)$, where $x$ is the index of $(A\upto p(n))$ in $D_{h(n)}$. By construction, $\abs{\sigma_n}=4\log{n}+\abs{\sigma}+2$, so for all sufficiently large $n$, we have $\abs{\sigma_n}<5\log{n}$.

Choose some $N$ such that $(A\upto p(N))\in D_{h(N)}$ and $\abs{\sigma_N}<5\log{N}$. By the definition of $e(\sigma)$, we have that
\begin{equation*}
f_{e(\sigma_N)}(e(\sigma_N))=\Psi^{U(\sigma_N)}(e(\sigma_N))=\Psi^{A\upto p(N)}(e(\sigma_N)).
\end{equation*}
However, $\Psi^A(e(\sigma_N))=g(e(\sigma_N))$, by our choice of $A$ and $\Psi$. Since we assumed $g(e)\ne f_e(e)$ for any $e\in M$ where $f_e(e)$ converges, we must have $\Psi^A(e(\sigma_N))\ne f_{e(\sigma_N)}(e(\sigma_N))$. Therefore,
\begin{equation*}
\Psi^{A\upto p(N)}(e(\sigma_N))\ne\Psi^A(e(\sigma_N))=g(e(\sigma_N)).
\end{equation*}
In other words, $p(N)\le\abrack*{e(\sigma_N),g(e(\sigma_N))}$. Since $\abs{\sigma_N}<5\log{n}$, this contradicts our definition of $p$; therefore, no total injection in $\mathcal{F}$ samples $P$ with positive upper density.
\end{proof}

Combining this with our previous observations, we have confirmed the result anticipated at the start of this section:

\begin{thm}
The following principles are equivalent over $\mathbf{RCA}_0$:
\begin{itemize}
\item $\mathbf{ID0}$,
\item $\mathbf{AVOID}$, and
\item $\mathbf{DNRW}\lor\mathbf{DOM}$.
\end{itemize}
\end{thm}

\addcontentsline{toc}{chapter}{References}

\bibliographystyle{plainnat}
\bibliography{idcPaper}

\begin{thebibliography}{20}
\providecommand{\natexlab}[1]{#1}
\providecommand{\url}[1]{\texttt{#1}}
\expandafter\ifx\csname urlstyle\endcsname\relax
  \providecommand{\doi}[1]{doi: #1}\else
  \providecommand{\doi}{doi: \begingroup \urlstyle{rm}\Url}\fi

\bibitem[Astor(2015)]{dirPaper}
Eric~P. Astor.
\newblock Asymptotic density, immunity and randomness.
\newblock \emph{Computability}, 4\penalty0 (2):\penalty0 141--158, 2015.
\newblock ISSN 2211-3568.
\newblock \doi{10.3233/COM-150040}.

\bibitem[Astor et~al.(In preparation)Astor, Hirschfeldt, and
  Jockusch]{upperCones}
Eric~P. Astor, Denis~R. Hirschfeldt, and Carl~G. Jockusch, Jr.
\newblock Dense computability, upper cones, and minimal pairs.
\newblock In preparation.

\bibitem[Downey and Hirschfeldt(2010)]{dhBook}
Rodney~G. Downey and Denis~R. Hirschfeldt.
\newblock \emph{Algorithmic Randomness and Complexity}.
\newblock Theory and Applications of Computability. Springer, 2010.

\bibitem[Downey et~al.(2013)Downey, Jockusch, and Schupp]{DJSdensity}
Rodney~G. Downey, Carl~G. Jockusch, Jr., and Paul~E. Schupp.
\newblock Asymptotic density and computably enumerable sets.
\newblock \emph{Journal of Mathematical Logic}, 13\penalty0 (02), 2013.
\newblock ISSN 0219-0613.
\newblock \doi{10.1142/S0219061313500050}.

\bibitem[Downey et~al.(2015)Downey, Jockusch, McNicholl, and
  Schupp]{ershovDensity}
Rodney~G. Downey, Carl~G. Jockusch, Jr., Timothy~H. McNicholl, and Paul~E.
  Schupp.
\newblock Asymptotic density and the {E}rshov hierarchy.
\newblock \emph{Mathematical Logic Quarterly}, 61\penalty0 (3):\penalty0
  189--195, 2015.
\newblock ISSN 1521-3870.
\newblock \doi{10.1002/malq.201300081}.

\bibitem[Dzhafarov and Igusa(2017)]{robustCoding}
Damir~D. Dzhafarov and Gregory Igusa.
\newblock Notions of robust information coding.
\newblock \emph{Computability}, 6\penalty0 (2):\penalty0 105--124, 2017.
\newblock ISSN 2211-3568.
\newblock \doi{10.3233/COM-160059}.

\bibitem[Hirschfeldt et~al.(2016{\natexlab{a}})Hirschfeldt, Jockusch, Kuyper,
  and Schupp]{coarseReducibility}
Denis~R. Hirschfeldt, Carl~G. Jockusch, Jr., Rutger Kuyper, and Paul~E. Schupp.
\newblock Coarse reducibility and algorithmic randomness.
\newblock \emph{The Journal of Symbolic Logic}, 81\penalty0 (3):\penalty0
  1028--1046, 2016{\natexlab{a}}.
\newblock ISSN 0022-4812.
\newblock \doi{10.1017/jsl.2015.70}.

\bibitem[Hirschfeldt et~al.(2016{\natexlab{b}})Hirschfeldt, Jockusch,
  McNicholl, and Schupp]{coarseBound}
Denis~R. Hirschfeldt, Carl~G. Jockusch, Jr., Timothy~H. McNicholl, and Paul~E.
  Schupp.
\newblock Asymptotic density and the coarse computability bound.
\newblock \emph{Computability}, 5\penalty0 (1):\penalty0 13--27,
  2016{\natexlab{b}}.
\newblock ISSN 2211-3568.
\newblock \doi{10.3233/COM-150035}.

\bibitem[H{\"o}lzl et~al.(2017)H{\"o}lzl, Raghavan, Stephan, and
  Zhang]{weaklyRepresented}
Rupert H{\"o}lzl, Dilip Raghavan, Frank Stephan, and Jing Zhang.
\newblock \emph{Weakly Represented Families in Reverse Mathematics}, pages
  160--187.
\newblock Springer International Publishing, Cham, 2017.
\newblock ISBN 978-3-319-50062-1.
\newblock \doi{10.1007/978-3-319-50062-1_13}.

\bibitem[Igusa(2013)]{igusaNoMinimalPair}
Gregory Igusa.
\newblock Nonexistence of minimal pairs for generic computability.
\newblock \emph{The Journal of Symbolic Logic}, 78\penalty0 (2):\penalty0
  511--522, 2013.
\newblock ISSN 0022-4812.
\newblock \doi{10.2178/jsl.7802090}.

\bibitem[Jockusch(1973)]{upwardClosure}
Carl~G. Jockusch, Jr.
\newblock Upward closure and cohesive degrees.
\newblock \emph{Israel Journal of Mathematics}, 15\penalty0 (3):\penalty0
  332--335, 1973.
\newblock ISSN 0021-2172.
\newblock \doi{10.1007/BF02787575}.

\bibitem[Jockusch and Schupp(2012)]{JSgc}
Carl~G. Jockusch, Jr. and Paul~E. Schupp.
\newblock Generic computability, {T}uring degrees, and asymptotic density.
\newblock \emph{Journal of the London Mathematical Society}, 85\penalty0
  (2):\penalty0 472--490, 2012.
\newblock ISSN 0024-6107.
\newblock \doi{10.1112/jlms/jdr051}.

\bibitem[Jockusch and Stephan(1993)]{nonhighCohesive}
Carl~G. Jockusch, Jr. and Frank Stephan.
\newblock A cohesive set which is not high.
\newblock \emph{Mathematical Logic Quarterly}, 39\penalty0 (1):\penalty0
  515--530, 1993.
\newblock ISSN 1521-3870.
\newblock \doi{10.1002/malq.19930390153}.

\bibitem[Kapovich et~al.(2003)Kapovich, Myasnikov, Schupp, and
  Shpilrain]{genericComplexity}
Ilya Kapovich, Alexei Myasnikov, Paul~E. Schupp, and Vladimir Shpilrain.
\newblock Generic-case complexity, decision problems in group theory, and
  random walks.
\newblock \emph{Journal of Algebra}, 264\penalty0 (2):\penalty0 665--694, 2003.
\newblock ISSN 0021-8693.
\newblock \doi{10.1016/S0021-8693(03)00167-4}.

\bibitem[Kjos-Hanssen et~al.(2011)Kjos-Hanssen, Merkle, and
  Stephan]{complexityRecursion}
Bj{\o}rn Kjos-Hanssen, Wolfgang Merkle, and Frank Stephan.
\newblock Kolmogorov complexity and the {R}ecursion {T}heorem.
\newblock \emph{Transactions of the American Mathematical Society},
  363\penalty0 (10):\penalty0 5465--5480, 2011.
\newblock ISSN 0002-9947.
\newblock \doi{10.1090/S0002-9947-2011-05306-7}.

\bibitem[Martin(1966)]{martin66}
Donald~A. Martin.
\newblock Classes of recursively enumerable sets and degrees of unsolvability.
\newblock \emph{Mathematical Logic Quarterly}, 12\penalty0 (1):\penalty0
  295--310, 1966.
\newblock ISSN 1521-3870.
\newblock \doi{10.1002/malq.19660120125}.

\bibitem[Miller and Nies(2006)]{permutationRandomness}
Joseph~S. Miller and Andr{\'e} Nies.
\newblock Randomness and computability: Open questions.
\newblock \emph{The Bulletin of Symbolic Logic}, 12\penalty0 (3):\penalty0
  390--410, 2006.
\newblock ISSN 1079-8986.
\newblock \doi{10.2178/bsl/1154698740}.

\bibitem[Nies et~al.(2005)Nies, Stephan, and Terwijn]{randomDegrees}
Andr{\'e} Nies, Frank Stephan, and Sebastiaan~A. Terwijn.
\newblock Randomness, relativization, and turing degrees.
\newblock \emph{The Journal of Symbolic Logic}, 70\penalty0 (2):\penalty0
  515--535, 2005.
\newblock ISSN 0022-4812.
\newblock \doi{10.2178/jsl/1120224726}.

\bibitem[Simpson(2009)]{SOSOA}
Stephen~G. Simpson.
\newblock \emph{Subsystems of second order arithmetic}.
\newblock Perspectives in Logic. Cambridge University Press, 2009.

\bibitem[Soare(2016)]{cta}
Robert~I. Soare.
\newblock \emph{Turing Computability: Theory and Applications}.
\newblock Theory and Applications of Computability. Springer-Verlag, 2016.

\end{thebibliography}
\end{document}